\documentclass{amsart} 
\usepackage{amssymb} 
\usepackage[utf8]{inputenc}
\usepackage[english]{babel}
\usepackage{lineno,hyperref}
\usepackage{amsmath}
\usepackage{amscd}
\usepackage{amsthm}
\usepackage{amsfonts}
\usepackage{color}
\usepackage{xfrac}
\usepackage{mathtools}
\usepackage{graphicx}
\usepackage{diagrams}  
\usepackage{tikz}
\newtheorem{theorem}{Theorem}[section]
\newtheorem{proposition}[theorem]{Proposition}
\newtheorem{corollary}[theorem]{Corollary}
\newtheorem{lemma}[theorem]{Lemma}
\newtheorem{remark}[theorem]{Remark}

\newcommand\restr[2]{\ensuremath{\left.#1\right|_{#2}}}

\begin{document}  

\title{Controlled surgery and $\mathbb{L}$-homology}
 
\author[F. Hegenbarth and D. Repov\v{s}]{Friedrich Hegenbarth and Du\v{s}an Repov\v{s}}

\address{Dipartimento di Matematica "Federigo Enriques", 
Universit\` a degli studi di Milano, 
20133 Milano, Italy}
\email{friedrich.hegenbarth@unimi.it} 

\address{Faculty of Education and Faculty of Mathematics and Physics, University of Ljubljana \& Institute of Mathematics, Physics and Mechanics,
 1000 Ljubljana, Slovenia}
\email{dusan.repovs@guest.arnes.si}

\begin{abstract}
This paper presents an alternative approach to controlled surgery obstructions. The obstruction for a degree one normal map $(f,b): M^n \rightarrow X^n$ with control map $q: X^n \rightarrow B$ to complete controlled surgery is an element $\sigma^c (f, b) \in H_n (B, \mathbb{L})$, where $M^n, X^n$ are topological manifolds of dimension $n \geq 5$. Our proof uses essentially the geometrically defined $\mathbb{L}$-spectrum as described by Nicas (going back to Quinn) and some well
known homotopy theory.
We also outline the construction of
the algebraically defined obstruction, and  we explicitly describe the assembly map $H_n (B, \mathbb{L}) \rightarrow L_n (\pi_1 (B))$ in terms of forms
in the case $n \equiv 0 (4)$.
 Finally, we explicitly  determine
 the canonical map $H_n (B, \mathbb{L}) \rightarrow H_n (B, L_0)$.
\end{abstract}

\keywords{Generalized manifold,
resolution obstruction, 
controlled surgery, 
controlled structure set,
 $\mathbb{L}_q$-surgery, 
Wall obstruction}
\subjclass[2010]{Primary 
57R67,
57P10,
57R65; 
Secondary 
55N20, 
55M05}
\dedicatory{Dedicated to the memory of Professor Andrew Ranicki (1948-2018)}

\date{\today}
\maketitle

\section*{Introduction}
To solve a surgery problem one encounters an obstruction being an element of the Wall
group~\cite{Wall70}. If one does controlled surgery with respect to a control
map over $B$, the obstruction belongs to a controlled version of Wall
groups. Both groups are constructed in a purely algebraic way as equivalence classes of certain forms or formations. The principal result (cf.~Theorem~\ref{theorB}~in~Section~\ref{sec3}) of the present
 paper shows
  that 
  controlled obstructions are elements of $H_n (B, \mathbb{L})$, where $\mathbb{L}$ is the geometrically defined surgery spectrum as described by Nicas~\cite{Nic82}. The basic idea of our
   proof is that controlled surgeries are done in small regions of the manifold when projecting it 
   onto
    $B$
  (and this fits well with  $\mathbb{L}$-homology
   of $B$).
   The proof is given in Section~\ref{sec3}.

In Section~\ref{sec1} we review the algebraic construction of controlled surgery obstructions for the case $n \equiv 0 (4)$ in terms of forms. In Proposition~\ref{prop1.1} we
 show how to obtain from this the Hermitian form of the uncontrolled surgery 
obstruction.

In~Section~\ref{sec2} we introduce  relevant surgery spaces and $\mathbb{L}$-spectra. We follow the Nicas description~\cite{Nic82} (which goes back to Quinn~\cite{Qui69}). The surgery spaces and spectra are defined semi-simplicially, i.e. by adic surgery problems. According to the targets of the surgery problems, one obtains spectra denoted by $\mathbb{L}$, resp. $\mathbb{L}^{PD}$. Here,
 the targets in $\mathbb{L}^{PD}$ are adic 
Poincar\'{e} duality complexes, whereas in $\mathbb{L}$
they
are adic
manifolds. 

Then we prove that the natural inclusion $\mathbb{L} \rightarrow \mathbb{L}^{PD}$ is a homotopy equivalence
 (cf.~Proposition~\ref{lem:equivalence}).
 In particular, $$\pi_n (\mathbb{L}) \cong \pi_n (\mathbb{L}^{PD}),$$
 and as shown by~Wall~\cite{Wall70},
 $$\pi_n (\mathbb{L}^{{PD}}) \cong L_n (\{1\}),$$
  the Wall-group of the trivial group. 
  We note that this problems was not addressed by Nicas~\cite{Nic82}. 

In Section~\ref{subsec2.2} we describe elements of the $\mathbb{L}$-homology group.
The spectrum $\mathbb{L}$ is not connected, in fact, $$\pi_0 (\mathbb{L}) = L_0 \cong \mathbb{Z}.$$
 There is a fiber
sequence of spectra 
$$\mathbb{L} <1> \rightarrow \mathbb{L} \rightarrow K (\mathbb{Z}, 0)$$ 
with $\mathbb{L}<1>$ 
the connected covering of $\mathbb{L}$, and $K (\mathbb{Z}, 0)$ the Eilenberg Mac-Lane spectrum.
We study the induced map $$H_n (B, \mathbb{L}) \rightarrow H_n (B, L_0)$$ and give an explicit formula in Section~\ref{subsec2.3} (cf.~Corollary~\ref{cor:2.4}). It has particular significance when determining the resolution invariant of Quinn (\cite{Qui83, Qui87}).

In Section~\ref{sec3} we treat
 $H_n (B, \mathbb{L})$ as the controlled Wall  group 
 and we present the main result of this paper - an alternative proof
 % of a theorem which was first proved by Ferry~\cite[Theorem 1.1]{Fer10}, to the effect 
 that  $H_n(B,\mathbb{L})$ is the obstruction group
for controlled surgery problems (cf. Theorem~\ref{theorB}).
Finally, in Epilogue we discuss controlled Wall
realizations of elements in $H_{n+1} (B, \mathbb{L})$ on $n$-manifolds $X$.

\section{Controlled and uncontrolled surgery obstructions}\label{sec1}

\noindent{\bf I.} \indent In this section we denote by $B$ a finite connected polyhedron with fundamental group $\pi = \pi_1(B)$, giving rise to the group ring $\Lambda = \mathbb{Z}[\pi]$. We shall
restrict ourselves
 only to the oriented situation, i.e.
when
 the usual orientation map $\pi  \rightarrow    \{ \pm 1 \}$ is $1$. 
More precisely, we shall
work in the category of oriented topological manifolds and topological bundles.
Normal degree one maps 
$$(f,b): M^n  \rightarrow    X^n$$
 are defined as in Wall~\cite{Wall70} (here, $M$ in
$X$ are $n$-manifolds, possibly with nonempty boundary $\partial M$ and 
$\partial X$, respectively).

 We add to this a reference map $q:X  \rightarrow     B$. In the controlled case it serves as the control
map, where $B$ is 
equipped
by
 a metric given by an embedding $B \subset \mathbb{R}^m$ as a
subcomplex, for a  sufficiently large
$m$.
For controlled surgery we assume that $q$ is a $UV^1$-map, i.e.
 for each contractible open set $U \subset B$, $\pi_1 (q^{-1}(U))=0$ (cf. e.g., Ferry~\cite{Fer94}). 

For
 $\dim X \geq 5$, it was proved by Bestvina that $q$ is homotopic to a $UV^1$-map (cf.~\cite[Theorem~4.4]{BFMW96}). In the case 
when
$\partial X \neq 0$, one must also
 assume that 
$$\restr{q}{\partial X}:\partial X \rightarrow B \ \ \hbox{\rm  is} \ \  UV^1,$$ so in this case
one
 must have
$\dim X \ge 6$. Suppose that
 $f$ restricts to a simple homotopy equivalence on the boundary
 $\partial X$. The map $f$
  can be made highly connected.

 In order to complete the surgery in the middle dimension, a surgery obstruction $\sigma (f,b),$ belonging to the Wall  group $L_n (\pi),$ must vanish.
Here, we may
assume without loss of generality that 
$$q_*: \pi_1 (X) \xrightarrow{\cong} \pi_1 (B).$$
Of course, this holds if $q$ is $UV^1$. If $\sigma (f,b) = 0$, then we get a simple homotopy equivalence $M'  \rightarrow    X$ relative the boundary, if $n\geq 5$, which is normally cobordant to $M  \rightarrow    X$. 

Controlled surgery is much more delicate (cf. \cite{book}). One can define an obstruction $\sigma^c (f,b),$ belonging to the controlled Wall group $L_n (B, \varepsilon, \delta)$  (in the notations of Pedersen, Quinn and Ranicki~\cite{PedQuiRan03}). Here, $\varepsilon >0$ is smaller than a certain $\varepsilon_0 > 0$ which depends on $B$ and $\dim X$,
and
$\delta >0$ is determined by $\varepsilon$. 

When $q$ is $UV^1$ and $n \geq 4$, the following holds: 
If $\sigma^c (f,b) = 0$
 then $(f,b): M  \rightarrow    X$ is normally cobodant to a $\delta$-homotopy equivalence $M' \xrightarrow[]{f'} X$ over $B$.
The map $f': M' \rightarrow X$ is unique up to $\varepsilon$-homotopy. 

This means that there exist a homotopy inverse $g' : X  \rightarrow     M$
 and homotopies $$h_t : f' \circ g' \sim Id_X, \ \  g_t : g' \circ f' \sim Id_{M'}$$ 
such that the tracks of the homotopies 
$$q \circ h_t, \ \ q \circ f' \circ g_t$$ are smaller than $\delta$, measured  in the metric of $B$. 
If $\partial X \neq \emptyset$, one has 
to
additionally assume that $\restr{f}{\partial M}$ is already a $\delta$-homotopy equivalence, and $f'$ is then a $\delta$-homotopy equivalence relative the boundary.

There is an obvious morphism 
$$L_n (B, \varepsilon, \delta)  \rightarrow     L_n (\pi),$$ forgetting the control, also considered as the assembly map. This is because controlled surgeries are done in small pieces which can be glued together to obtain the global result. We shall come back to this point in Section~\ref{sec3}. 

Here, we point out how one can obtain the Wall  obstruction $\sigma (f,b)$ from the controlled obstruction $\sigma^c (f,b)$ (cf. Part IV below).
We shall do this for $n \equiv 0(4)$. This is the case which is interesting for the resolution obstruction. 

\noindent{\bf  II.} \indent Let now 
 $n = 2k$,
  where $k$ 
is even. If $f: M  \rightarrow    X$ is highly connected
then
 one is left with
the following
 exact sequence
$$0  \rightarrow    K_k (f, \Lambda)  \rightarrow     H_k (M, \Lambda)  \rightarrow     H_k (X, \Lambda)  \rightarrow     0.$$
By duality and the Hurewicz-Whitehead theorems, one has to kill
 $$K_k (f, \Lambda) \cong \pi_{k+1} (X, M)$$ by surgeries. Here,
$K_k (f, \Lambda)$ is a stably free 
based $\Lambda$-module, finitely  generated, and carrying a Hermitian $\Lambda$-bilinear form $$\lambda: K_k (f,\Lambda) \times K_k (f,\Lambda)   \rightarrow     \Lambda$$ which is refined by a quadratic form $\mu$, deduced from the bundle map $b$. In Wall~\cite[p.~47]{Wall70}, 
this is called a special Hermitian form. 
Equivalence classes of such special Hermitian forms constitute the Wall group $L_{2k} (\pi)$ (cf. Wall~\cite[Chapter~5]{Wall70} for precise constructions). Hence $$\sigma (f,b) = [K_k (f,\Lambda), \lambda, \mu] \in L_{2k} (\pi).$$

\noindent {\bf III.} \indent We are
now  going to describe the controlled surgery obstructions. It was Quinn who explicitly constructed them (cf. Quinn~\cite[Section~2]{Qui83}). His aim was to prove the
existence of resolutions of generalized manifolds. For this purpose it was not necessary to construct controlled Wall  groups (cf.  also Quinn~\cite{Qui87}). 
A detailed construction can be found in Ferry~\cite{Fer10}. To obtain controlled results one has to work with the chain complex $C_{\#} (X, M)$ instead of homology. Here are the main steps: 
\begin{enumerate}
	\item[Step 1.] $(f,b): M  \rightarrow    X$ is normally cobordant to $(\overline{f}, \overline{b}): \overline{M}  \rightarrow    X$ so that $C_g (X, \overline{M}) = 0$ for $j \leq k$. 
This can be obtained for any surgery problem.	
	To continue,
 we recall that manifolds  $\overline{M}$
satisfy the
controlled Poincar\' e
duality, i.e. the cap product with a fundamental cycle is a $\delta$-chain equivalence $C^{\#} (\overline{M})  \rightarrow     C_{n-\#} (\overline{M})$, and this implies a $\delta$-chain
equivalence $$C^{\#} (X, \overline{M})  \rightarrow     C_{n+1-\#} (X, \overline{M})$$ for arbitrary $\delta > 0$. 
	\item[Step 2.] Using the $\delta$-chain
equivalence 
$$C^{\#} (X, \overline{M})  \rightarrow     C_{n+1-\#} (X, \overline{M})$$ and controlled cell
trading, one proves that $C_{\#} (X, \overline{M})$ is $\delta$-chain equivalent to a chain complex of the type 
$$0  \rightarrow     D_{k+1}  \rightarrow     D_k  \rightarrow     0.$$
	
By doing surgery on small $k$-spheres in $\overline{M},$ according to the
basis
 of $D_k$, one obtains a chain complex of the
type $$0  \rightarrow     A_{k+1}  \rightarrow     0.$$
 Let $M'$ be the result of this
surgery.

	\item[Step 3.] By Quinn~\cite[Proposition~2.4]{Qui83}, the pair $(X, M')$ is $\delta$-homotopy equivalent to a pair $(X', M')$ such that 
	$$ C_{\#} (X', M') =
  		\begin{cases}
    		A_{k+1}	& \quad \# = k+1\\
    		0  		& \quad \text{otherwise.}\\
  		\end{cases}
		$$
\end{enumerate}
Since the chain equivalence in
 Step~1 is a $\delta$-equivalence for arbitrary small $\delta$, we have the same situation in Step~3. So the composition 
$$q' : X'  \xrightarrow[]{\sim}    X \xrightarrow[]{g} B$$ is a $UV^1 (\delta)$-map. This will be sufficient for our purpose (cf.  e.g., Ferry~\cite{Fer10},
Quinn~\cite{Qui83},
Yamasaki~\cite{Yam87} for the concept of geometric algebra of chain complexes,
 $UV^1 (\delta)$, and $\delta$-chain equivalences).

By Step~3, our original surgery problem $M  \rightarrow    X$ is replaced by a normal degree one map $$(f',b'): M'  \rightarrow    X',$$
 where $b'$ is a
 bundle map between the normal bundle $\nu_{M'}$ of $M'$ and the bundle $\xi$ over $X'$, induced by the map $X'  \rightarrow    X$ from the normal bundle $\nu_{X}$ of $X$.

The result is a finitely generated geometric $\mathbb{Z}$-module $C_{k+1} (X',M')$, with obvious intersection number $$\lambda_\mathbb{Z} : C_{k+1} (X',M') \times C_{k+1} (X',M')  \rightarrow    \mathbb{Z},$$ 
refined by a quadratic form $\mu_\mathbb{Z}$,
 determined by the normal data, such that the radius of $\lambda_\mathbb{Z}$ is $\delta$-small: for basis elements $$a,b \in C_{k+1} (X',M')$$ one has 
$$\lambda_\mathbb{Z} (a,b) = 0 \ \  \hbox{\rm  provided that} \ \ d (q'(a), q'(b)) > \delta.$$

The equivalence class of $$[C_{k+1} (X',M'), \lambda_\mathbb{Z}, \mu_\mathbb{Z}] \in L_n(B, \varepsilon, \delta)$$ is the controlled surgery obstruction of the surgery problem $$(f',b'): M'  \rightarrow    X'.$$ 

One notes that the Wall  obstructions $\sigma (f,b)$ and $\sigma (f',b')$ in $L_n (\pi)$ coincide.

\noindent {\bf IV.} \indent The map $L_n (B, \varepsilon, \delta)  \rightarrow    L_n (\pi).$\\
We are given $\sigma (f,b) \in L_n (\pi)$ which we represent by the triple
 $(K_k (f', \Lambda), \lambda, \mu)$. One first notes  that $$K_k (f', \Lambda) = C_{k+1} (X', M') \otimes_\mathbb{Z} \Lambda.$$ Let
$$a_1, \ldots a_r \in C_{k+1} (X', M')$$ be 
a
 $\mathbb{Z}$-basis.
Then 
$$\widetilde{a}_i = a_i \otimes 1, \ \  i = 1, \ldots r$$ is a $\Lambda$-basis of $K_k (f', \Lambda)$. To calculate $\lambda_\mathbb{Z} (a_i, a_j)$, one 
observes
 that the
$a_i$'s are represented by small maps 
$$(D^{k+1}, S^k )  \rightarrow    (X', M'),$$
where
$ \partial a_i : S^k   \rightarrow     M'$ are framed immersions in general position. Let 
$$\partial a_i \cap \partial a_j = \{ p_1, \ldots , p_m \}.$$
Then $$\lambda_\mathbb{Z} (a_i, a_j) = \sum_{i=1}^m \varepsilon_i, \ \ \hbox{\rm where} \ \ \varepsilon_i = \pm 1$$ is the usual algebraic intersection number at the point $p_i$.

The elements $$\widetilde{a}_1, \ldots , \widetilde{a}_r \in K_k (f', \Lambda) \cong C_{k+1} (\widetilde{X'}, \widetilde{M'})$$ are considered as liftings of $\partial a_1, \ldots , \partial a_r$ in the universal covering $\widetilde{M'}$ of $M'$. Alternatively,
 $\widetilde{a}_1, \ldots , \widetilde{a}_r$ are immersed spheres in $M'$ together with connecting paths to a base point of $M'$. We state our observation in the following

\begin{proposition}\label{prop1.1}
With the above assumptions and notations we have
$$\lambda (\widetilde{a}_i, \widetilde{a}_j) = \lambda_\mathbb{Z} (a_i, a_j) g_{ij} \in \Lambda,$$
where $g_{ij} \in \pi$ is determined by the paths connecting $\widetilde{a}_i, \widetilde{a}_j$ to the base point. 
\end{proposition}

\begin{proof}
Since the radius of $\lambda_\mathbb{Z}$ is as small as we want, and the immersed spheres are small, we may assume that their
 images in $B$ are contained in a contractible subset. By the $UV^1$ property we conclude that $$\widetilde{a}_i (S^k) \cup \widetilde{a}_j (S^k) \subset U \subset M' \ \ \hbox{\rm with} \ \ \pi_1 (U) = \{1\}.$$
Calculating $\lambda (\widetilde{a}_i, \widetilde{a}_j)$ as in the proof in Wall~\cite[Theorem~5.2]{Wall70}, one obtains the claim.
\end{proof}

The case when $\pi$ is the fundamental group of the $n$-torus, this was first proved by Mio and Ranicki~\cite[Section~10.1]{MioRan06}.
Since any surgery problem $(f,b) : M^n   \rightarrow     X^n$ between $n$-manifolds without boundaries can be considered as a
controlled problem over $Id: X  \rightarrow    X$, we can get the following

\begin{corollary}
Let $n \equiv 0(4)$. Then $$\sigma (f,b) \in L_n (\pi_1 (X))$$ has a representation $(G, \lambda, \mu)$ with $G$ a free $\Lambda$-module with basis $b_1, \ldots , b_r$ such that $$\lambda (b_i, b_j) = n_{ij} g_{ij}, \  n_{ij} \in \mathbb{Z} ,\ \hbox{\rm and} \ g_{ij} \in \pi_1 (X).$$
\end{corollary}

\begin{remark}
If $\partial M, \partial X$ are nonempty, the restriction $\restr{f}{\partial M}$ has to be a $\delta$-controlled homotopy equivalence. In the case of $Id: X  \rightarrow    X$ as the control map this implies that $\restr{f}{\partial M}$ is a homeomorphism. However,
 if $\restr{f}{\partial M}$ is a $\delta$-homotopy equivalence for some $UV^1$-map $q: X  \rightarrow    B,$ then the proof goes through.
\end{remark}

\section{\texorpdfstring{$\mathbb{L}$}{L}-spectra and \texorpdfstring{$\mathbb{L}$}{L}-homology}\label{sec2}

\subsection{On the geometric construction of the \texorpdfstring{$\mathbb{L}$}{L}-spectrum}\label{subsec2.1}

The geometric $\mathbb{L}$-spectrum was introduced in  Quinn~\cite{Qui69} as a semi-simplicial $\Omega$-spectrum. Details can also be found in
Nicas~\cite{Nic82} which we
shall
 follow. 
We
 define surgery spaces $\mathbb{L}_r (B)$,  where $B$ is
a polyhedron. We are only interested in
the case  $B = \{*\}$
 and we 
shall
write $\mathbb{L}_r = \mathbb{L}_r \{*\}$.

An $s$-simplex $\sigma \in \mathbb{L}_r$ is a normal degree one map between $(r+s)$-dimensional oriented $(s+3)$-ads of manifolds 
$$(M, \partial_0 M, \ldots, \partial_s M, \partial_{s+1} M)   \rightarrow     (X, \partial_0 X, \ldots, \partial_s X, \partial_{s+1} X)$$ such that $f$ restricted to $\partial_{s+1} M$ is a homotopy equivalence. To
each
 $\sigma$ belongs a reference map of $(s+3)$-ads $$(X, \partial_0 X, \ldots, \partial_s X, \partial_{s+1} X)   \rightarrow     (\Delta^s, \partial_0 \Delta^s, \ldots, \partial_s \Delta^s, \Delta^s)$$ to the standard $s$-simplex
$\Delta^s$. Note that the last face $\partial_{s+1} X$ maps to the interior of $\Delta^s$, and plays a special role in the constructions. 

Let $\mathbb{L}_r (s)$ be the set of $s$-simplices. 
Then
$\mathbb{L}_r$ is a pointed semisimplicial complex with base points the empty problem and there is a homotopy equivalence to
the simplicial loop space of $\mathbb{L}_{r-1}$ (cf. Nicas~\cite[Proposition~2.2.2]{Nic82}):
$$\mathbb{L}_r   \rightarrow     \Omega \mathbb{L}_{r-1}.$$

The collection of surgery spaces 
$\{ \mathbb{L}_r, r \in \mathbb{Z} \}$ 
defines a spectrum 
$\mathbb{L}^{+}$
such that its homotopy groups 
$\pi_n (\mathbb{L}^{+})$ 
are the Wall  groups 
$L_n (1)$.
In the notation of~\cite{Ran92},
$\mathbb{L}^{+}=\mathbb{L}<1>$, whereas
$\mathbb{L}$
denotes the periodic $\mathbb{L}$-spectrum with the $0$-term $=
 \mathbb{Z} \times \sfrac{G}{TOP}.$

In order to do this we have to address two problems. 
The first one
comes from the following easily proved (and well
known) lemma.

\begin{lemma}\label{lem3.1}
The surgery space $\mathbb{L}_0$ defined above
satisfies $\pi_0 (\mathbb{L}_0) = \{0\}$. 
\end{lemma}

\begin{proof}
Recall, that we are working in the simplicial category. A typical element $\sigma \in \mathbb{L}_0 (0)$ is a map of degree one of
the
 type $\{ \pm y_1, \ldots, \pm y_k \}   \rightarrow     \{ x \}$.
By the degree one property one can reorder it as
follows
 $$\{  y_1, + y_2, - y_2, \ldots, + y_l, - y_l \}   \rightarrow     \{ x \}.$$ 
The $1$-simplex $\{ I_1, \ldots , I_l \}   \rightarrow     J$, with $I_j$ denoting
the interval with $\partial I_j = \{ y_j, -y_j \}$,
 shows that $\sigma$ is equivalent to $(\{ y_1 \}   \rightarrow     \{ x \})$.
Here we 
view
 $J$ as a degenerate  $1$-simplex consisting of a single
point. Moreover, $(\{ y \}   \rightarrow     \{ x \})$ is equivalent to the empty set. Therefore $\pi_0 (\mathbb{L}_0) = 0$.
\end{proof}

The second problem arises from comparison with the Wall groups in Wall~\cite[Chapter~9]{Wall70} (cf. the proof of Nicas~\cite[Proposition~2.2.4]{Nic82}). The point is that in Wall~\cite{Wall70}, Poincar\'{e}
duality spaces  are used 
as targets, 
whereas
in~\cite{Nic82} manifolds
are used.
This point was not addressed in Nicas~\cite{Nic82}. It might be not the same for a generic polyhedron $B$, but it gives the same result when $B = \{ * \}$.

To see this,
we introduce the surgery spaces $\mathbb{L}^{PD}_r$ in the same
 way as $\mathbb{L}_r$, but  Poincar\'{e}-ads as targets (this was used in Quinn~\cite{Qui69}).
One also proves that  $\mathbb{L}^{PD}_r$ is homotopy equivalent to  $\Omega \mathbb{L}^{PD}_{r-1}$. 
There is an obvious map 
$\mathbb{L}_r   \rightarrow     \mathbb{L}^{PD}_r$, 
and
 $$\pi_0 (\mathbb{L}_r)   \cong   \pi_0 (\mathbb{L}^{PD}_r) = \{ 0 \}.$$ 
We can define $\Omega$-spectra $\mathbb{L}^+$ and $\mathbb{L}^{PD}$ using this. 

To match up with the usual notation,
we write
 $$\mathbb{L}^+ = \{ \mathbb{L}_{-r}, r \geq 0 \}, \ \ 
 \mathbb{L}^{PD} = \{ \mathbb{L}^{PD}_{-r}, r \geq 0 \}.$$
Both spectra are connected and 
$\mathbb{L}^+$ 
 becomes 
$\mathbb{L}\langle 1 \rangle$ in the notations of Ranicki~\cite{Ran92}.

\begin{proposition}\label{lem:equivalence}
The map $\mathbb{L}^+   \rightarrow     \mathbb{L}^{PD}$ is a homotopy equivalence.
\end{proposition}

\begin{proof}
We shall show that the induced morphism
 $$
 \pi_n (\mathbb{L}^+) \to
 \pi_n (\mathbb{L}^{PD})
 $$
 is an isomorphism for $n \geq 0.$ The assertion will then follow by the Whitehead theorem.

Observe that $$\pi_n (\mathbb{L}^{PD}) 
\cong 
\pi_{n +r} (\mathbb{L}^{PD}_{-r}) 
\cong  
\pi_n (\mathbb{L}^{PD}_0)
\cong
\pi_{0} (\Omega^{n}\mathbb{L}^{PD}_{-n}).
$$ 
However, the last one coincides with the group
$L^{1}_{n}(\{*\}),$
considered by Wall \cite[Chapter 9]{Wall70}.
We begin with the higher dimensional case.

{\bf Case I: $n\ge 5$.}
Wall defines a restricted set 
$$
L^{2}_{n}(\{*\})
\subset
L^{1}_{n}(\{*\})
$$
consisting of simply-connected surgery problems (an adic version of this was considered by Nicas~\cite[Chapter 2]{Nic82}).
He shows that 
$$
L^{2}_{n}(\{*\})
\to
L^{1}_{n}(\{*\})
$$
is bijective for 
$n\ge 4$ (cf. Wall~\cite[Theorem 9.4]{Wall70}, for the adic case cf. Nicas~\cite[Proposition 2.2.7]{Nic82}).
A corollary of this is that the surgery obstruction map
$$
\Theta:
L^{1}_{n}(\{*\})
\to
L_{n} \ \ (  = \hbox{\rm{Wall group of}} \   \pi_1=\{1\}) 
$$
is an isomorphism for 
$n \ge 5$ (cf. \cite[Corollary 9.4.1.]{Wall70}).
Since the composition
$$
L_n=\pi_n (\mathbb{L}^+)
 \to
 \pi_n (\mathbb{L}^{PD})
 \cong
 L^{1}_{n}(\{*\})
 \xrightarrow[]{\Theta}
L_{n}
$$
is the identity, this proves that we indeed have an isomorphism 
$$
\pi_n (\mathbb{L}^+) 
\xrightarrow[]{\cong}
\pi_n (\mathbb{L}^{PD})
$$
for all $n\ge 5$.

{\bf Case II: $n=4$.}
The surgery obstruction map $\Theta$
is defined for $n=4$
and the composition 
$$
L_4=\pi_4 (\mathbb{L}^+)
 \to
 \pi_4 (\mathbb{L}^{PD})
 \cong
 L^{1}_{4}(\{*\})
 \xrightarrow[]{\Theta}
L_{4}
$$
is the identity. Therefore 
$$
\pi_4 (\mathbb{L}^+) 
\to 
\pi_4 (\mathbb{L}^{PD})
$$
is injective.
Since
$$
L^{2}_{4}(\{*\})
\xrightarrow[]{\cong}
L^{1}_{4}(\{*\}),
$$
we can represent an element in 
$ \pi_4 (\mathbb{L}^{PD})$
by
$$
(f,b):M\to X 
\ \
\hbox{\rm{with}} 
 \ \
 \pi_1(X)=\{1\}.$$
Assume first that $\partial X= \emptyset$.
Then
$G=H_2(X,\mathbb{Z})$
is $\mathbb{Z}$-free and the intersection form
$$
\lambda_X: G\times G\to \mathbb{Z}
$$
is unimodular.
By Freedman~\cite[Theorem 1.5]{Freedman},
there is a simply-connected 4-manifold $M'$
realizing 
$(G,\lambda_X).$
However, by Milnor~\cite{Milnor}, $M'$ is homotopically equivalent to $X$, therefore
$$
(f,b):M\to X 
$$
is equivalent to the surgery problem
$$
(f',b'):M\to M' 
$$
arising from 
$
\pi_4 (\mathbb{L}^+). 
$
Now assume that  $\partial X \neq \emptyset$. Then
$$
\restr{f}{\partial M}: \partial M \to \partial X 
$$
is a homotopy equivalence.
We obtain a closed surgery problem by glueing 
$$
Id:M\to M \ \ 
\hbox{\rm{and}}
\ \ 
f:M\to X
$$
along the boundary
$$
N=M \underset{Id}\cup M\xrightarrow[]{{Id}\cup f} M \underset{\restr{f}{\partial M}}\cup X=Y.
$$
By the van Kampen theorem,
$
 \pi_1(Y)=\{1\}.
$
It is now easy to see that the class of $N\to Y$
represents the same as the classes of
$$
(f,b):M\to X \ \ 
\hbox{\rm{and}} \ \ 
{Id}:M\to M
$$
in
$L^{1}_{4}(\{*\})$
(cf. Supplement below).
However, 
${Id}:M\to M$
represents the trivial class, so we are back in the closed case.

{\bf Case III: $n=3$.} (See also a short proof in Supplement below.)
Let 
$$
(f,b):M^3\to X^3 
$$
be
given. As in the case $n=4$, we may assume that $\partial X = \emptyset.$
There is a commutative diagram of well-known isomorphisms of Hurewicz maps between cobordism groups

\begin{center}
\begin{tikzpicture}[node distance=1.5cm, auto]
  \node (N1) {$\Omega_{3}(X)$};
  \node (M) [right of=N1] {};
  \node (N2) [right of=M] {$\Omega^{PD}_{3}(X)$};
  \node (X) [node distance=1.5cm, below of=M] {$H_{3}(X,\mathbb{Z})$};
  \draw[->, font=\small] (N1) to node {$\mu$} (N2);
  \draw[->, font=\small] (N1) to node {} (X);
  \draw[->, font=\small] (N2) to node {} (X);
\end{tikzpicture}
\end{center}
It follows that $\mu$ is an isomorphism and since
$f$ is of degree one, $M$ is $PD$-cobordant to $X$ over $X$.

Let $q:Z\to X$
be a $PD_{4}$-complex over $X$ with
$$\restr{q}{X}=Id \ \ 
\hbox{\rm{and}} \ \ 
\restr{q}{M}=f.$$
The Spivak fibration $\nu_Z$ of $Z$ restricts to the Spivak fibration $\nu_X$ and $\nu_M$, and we have the maps of the $m$-sphere into the Thom spaces

$$
(S^m \times I, S^m \times \{0\}, S^m \times \{1\})
\to
(T{\nu_Z}, T{\nu_X}, T{\nu_M}).
$$
Since $M$ is a manifold, let us for simplicity write
$\nu_M$ also for the stable normal bundle of $M\subset S^m$, i.e.
$$b:\nu_M \to \xi,$$
where $\xi$ is a certain topological
reduction of $\nu_X$.

{\bf Claim.} {\sl If $\nu_Z$ has a topological reduction $\omega$ which restricts to $\xi$ on $X$, then
$$(f,b):M\to X$$ is equivalent to a normal degree one map
$$(f'',b''): M'' \to M, \ 
\hbox{\rm{ where}} \ \ 
b'':\nu_{M''} \to \eta \ \  
\hbox{\rm{and}} \ \ 
\eta=\restr{\omega}{M}.$$}
This is obtained by taking the transverse inverse images of the composition of $(Z,X,M)$:
$$
(S^m \times I, S^m \times \{0\}, S^m \times \{1\})
\to
(T{\nu_Z}, T{\nu_X}, T{\nu_M})
\xrightarrow[]{h}
(T{\omega}, T{\xi}, T{\eta}),
$$
where $h$ comes from the reduction $\omega$ of $\nu_Z$.

Now, the obstructions to existence of such $\omega$
belong to
$$
H^{r+1}(Z,X,\pi_{r}(\sfrac{G}{TOP}))
$$
hence there is 
only
one in 
$$
H^{3}(Z,X,\pi_{2}(\sfrac{G}{TOP}))
\cong
H^{3}(Z,X,\mathbb{Z}_2).
$$
Since
$X\subset Z \xrightarrow[]{q} X$
is the identity, the homomorphism
$$
H^{r}(Z,\mathbb{Z}_2)
\to
H^{r}(X,\mathbb{Z}_2)
$$
is surjective, i.e. the short cohomology
 sequence
$$
0
\to 
H^{3}(Z,X,\mathbb{Z}_2)
\to 
H^{3}(Z,\mathbb{Z}_2)
\to
H^{3}(X,\mathbb{Z}_2)
\to
0
$$
is exact.

The image of the obstruction in
$
H^{3}(Z,\mathbb{Z}_2)
$
is 0 because $\nu_z$ has topological reduction (cf. Hambleton~\cite{Hambleton}).
Therefore such $\omega$ exists which proves the surjectivity of
$$
\{0\}=\pi_{3}(\mathbb{L}^{+})\to \pi_{3}(\mathbb{L}^{PD}),
$$
i.e. 
$\pi_{3}(\mathbb{L}^{PD})=\{0\}$.

{\bf Case IV: $n=1,2$.}
These two cases are obvious since
for $n=1,2$ all $PD$-complexes are manifolds.

This completes the proof of Proposition~\ref{lem:equivalence}.  
\end{proof}

{\bf Supplement.} We add two remarks here.

{\bf 1.}
In the case $n=4$ and $\partial X \neq \emptyset$, a normal cobordism between 
$$N=M\underset{Id}\cup M\to M\underset{\restr{f}{\partial M}}\cup X, \ \ 
M\xrightarrow[]{(f,b)} X, \ \
\hbox{\rm{and}} \ \ 
Id:M\to M
$$
can be constructed as follows:
replace $X$ by $$X'=X\underset{\restr{f}{\partial M}}\cup \partial M \times I$$
being homotopy equivalent to $X$ with a collared boundary
$\partial M \subset X'$.
Then glue
$$M\times I 
\overset{\cdot}{\cup} 
X'\times I \ \
\hbox{\rm{at}} \ \ 
M\times\{0\}\cup X'\times \{0\}$$
along the collar
$$\partial M \times [1-\varepsilon,1] \subset M \cap X'.$$ 
This gives a $PD_5$-complex $V^5$.
A similar
 construction on
$$M\times M 
\overset{\cdot}{\cup}
 M\times I$$ gives a 5-manifold $W^5$.
An obvious degree one normal map can be
constructed from $Id_M$
and $(f,b)$.
Note that 
$$\partial W=M
\overset{\cdot}{\cup} 
M 
\overset{\cdot}{\cup}
 M
 \underset{Id}\cup M \ \
 \hbox{\rm{and}}  \ \
 \partial V=X
 \overset{\cdot}{\cup} 
 M 
 \overset{\cdot}{\cup}
 M\underset{\restr{f}{\partial M}}\cup X.$$
 
 {\bf 2.} In the case $n=3$ it seems that one can replace
 the $PD_4$-complex $Z$ by $Z'$ with $\partial Z'=\partial Z$
 and $\pi_{1}(Z')=\{1\}$ by Poincar\'{e} surgeries.
 The obstruction to finding a reduction
 $\omega$ of $\nu_{Z'}$ such that 
 $$
 \restr{\omega}{X}=\xi \ \ 
\hbox{\rm{and}} \ \ 
 \restr{\omega}{M}=\nu_M
 $$
 belongs to
 $$
 H^{3}(Z', M
 \overset{\cdot}{\cup}
  X, L_2)
 \cong
 H_1(Z',L_2)=0.
 $$ 
Then we get a normal bordism between
 $$(f,b):M\to X 
\ \  
 \hbox{\rm{and}}
  \ \
  Id:M\to M,
  $$
  hence the class of $(f,b)$
  is trivial. 

\subsection{Concerning the elements of \texorpdfstring{$H_n (B, \mathbb{L})$}{Hn (B,L)}}\label{subsec2.2}

We shall write as before $\mathbb{L}$ for the periodic spectrum $\mathbb{L} \langle 0 \rangle$, and $\mathbb{L}^+ = \mathbb{L} \langle 1 \rangle$ for its connective 
covering spectrum. Recall the fibration
sequence (cf. Ranicki~\cite[Section~15]{Ran92})
$$\mathbb{L}^+   \rightarrow     \mathbb{L}   \rightarrow     K (L_0, 0),$$
where $K (L_0, 0)$ is the Eilenberg-MacLane spectrum. We shall study the homology of this sequence in 
Subsection~\ref{subsec2.3}.

Here, we want to describe elements $x \in H_n (B, \mathbb{L})$, where $B \subset S^m$ is a finite polyhedron. We follow Ranicki~\cite[Section~12]{Ran92}, to represent $x$ by a cycle, using a dual cell
decomposition of $S^m$. This is justified by Ranicki~\cite[Remark~12.5]{Ran92}.

If $\sigma$ is a simplex of $S^m$, let $D(\sigma, S^m)$ be its dual cell. It has a canonical $(m - |\sigma| + 3)$-ad structure, where $|\sigma| = \dim\sigma$ 
and 
 $$m - |\sigma| = \dim D (\sigma, S^m).$$

The element $x$ is then represented by a simplicial map $$(S^m, S^m \setminus B)   \rightarrow     (\mathbb{L}_{n-m}, \emptyset)$$ (one should merely replace $S^m \setminus B$ with the supplement of $B$, as done in Ranicki~\cite{Ran92}).
Let us first consider the case when
 $$x: (S^m, S^m \setminus B)   \rightarrow     (\mathbb{L}^+_{n-m}, \emptyset)$$ 
represents an element of
 $H_n (B, \mathbb{L}^+)$, i.e. $$x(\sigma) \in \mathbb{L}^+_{n-m} (m - |\sigma|).$$
However, this is the  surgery space described 
above, i.e. $x(\sigma)$ is a degree one normal map 
$$(f_\sigma, b_\sigma): M^{n-|\sigma|}_\sigma   \rightarrow     X^{n-|\sigma|}_\sigma$$ 
between $(n - |\sigma|)$-dimensional $(m - |\sigma| + 3)$-ads with a reference map $X^{n-|\sigma|}_\sigma \rightarrow     D(\sigma, S^m)$. The cycle 
 condition implies that they can be assembled (the colimit) to a degree one normal map $(f,b): M^n   \rightarrow     X^n$ with boundaries $\partial M, \partial X$, so
 that $\restr{f}{\partial M}$ is a homotopy equivalence, together with a reference map $X   \rightarrow     B$. Note that $x(\sigma) = \emptyset$ if $\sigma \notin B$, and 
$X\rightarrow B$ is the colimit of all
$$X^{n-|\sigma|}_\sigma   \rightarrow     D(\sigma, S^m) \subset S^m$$ 
with a retraction onto $B$ (cf. Nicas~\cite[Theorem~3.3.2]{Nic82}, or Laures and McClure~\cite[Proposition~6.6]{LauMcL14}). Moreover, the boundary map $\partial M   \rightarrow     \partial X$ is the colimit of the various homotopy equivalences
 $$\partial_{m - |\sigma| + 1} M^{n-|\sigma|}_\sigma   \rightarrow     \partial_{m - |\sigma| + 1} X^{n-|\sigma|}_\sigma.$$

To consider the general case $x \in H_n (B, \mathbb{L})$ we recall two properties:
\begin{enumerate}
	\item[(a)] (Periodicity): Suppose
that
 $\dim B - 1 \leq r$.
Then there is a natural isomorphism $H_r (B, \mathbb{L})   \rightarrow     H_{r+4} (B, \mathbb{L})$ (cf. Ranicki~\cite[p.~289-290]{Ran92});
	\item[(b)] If $\dim B < r$, then $H_r (B, \mathbb{L}^+) \xrightarrow{\cong} H_r (B, \mathbb{L}).$
\end{enumerate}

Both properties also easily follow from the Atiyah-Hirzebruch spectral sequence $$H_p (B, \pi_q (\mathbb{L})) \xrightarrow{p + q=r} H_r (B, \mathbb{L}),$$ and the periodicity of the $\mathbb{L}$-spectrum: 
$$\mathbb{L}_r \cong \mathbb{L}_s \ \  \hbox{\rm if} \ \  r - s \equiv 0 (4).$$

In order to represent $x \in H_n (B, \mathbb{L})$, we choose $r$ sufficiently large with $r - n \equiv 0 (4)$, and represent $x$ as 
an
element of
 $H_r (B, \mathbb{L}) \cong H_r (B, \mathbb{L}^+)$ as above. Assembling (colimit) then gives a degree one normal map $(f,b): P^r   \rightarrow     Q^r$ with the
reference map $q: Q^r   \rightarrow     B$, and $\restr{f}{\partial P}$ a homotopy equivalence.

A specific construction of the degree one normal map $P^r   \rightarrow     Q^r$ is given using the identification $H_n (B, \mathbb{L})$ with the controlled Wall group $L_n (B, \varepsilon, \delta)$,
 as established   by
 Pedersen, Quinn and Ranicki~\cite{PedQuiRan03}. 
 %So $x \in H_n (B, \mathbb{L})$ gives an element of  $L_n (B, \varepsilon, \delta)$.  
 Here are some details.
Suppose that also $n \equiv 0 (4)$.
Then $x$ corresponds to a triple $\{ G, \lambda_\mathbb{Z}, \mu_\mathbb{Z}\}$ as described in Section  \ref{sec1}. It can be considered as an
element 
of
 $L_r (B, \varepsilon, \delta)$ by the periodicity, $r - n  \equiv 0 (4)$, and 
it can be realized in a
controlled way,  in the sense of Wall on the boundary $\partial N$ of 
a
regular neighbourhood $N \subset \mathbb{R}^r$ of $B \subset \mathbb{R}^r$.

We obtain $P^r_0$ which can be written as $$P^r_0 = N \cup \partial N \times I \cup \{ \cup_k D^{\frac{r}{2}} \times D^{\frac{r}{2}} \}.$$
Here, $k = \hbox{\rm rank} \ G$, and $\lambda_\mathbb{Z}, \mu_\mathbb{Z}$ are realized as framed immersions $$S^{\frac{r}{2}} \times I    \rightarrow     \partial N \times I.$$ The handles $D^{\frac{r}{2}} \times D^{\frac{r}{2}}$ are attached 
to
 the top along the framed embeddings. By the controlled Hurewicz-Whitehead theorem and the $\alpha$-approximation theorem one gets a degree one normal map $P^r_0    \rightarrow     N$ of $r$-manifolds with boundary, such that $\partial P^r_0    \rightarrow     \partial N$ is a homeomorphism. 
Then we can close this in the usual way to get $$P^r = P^r_0 \cup_\partial N   \rightarrow     N \cup_\partial N = Q^r.$$

It is more convenient to consider $P^r_0   \rightarrow     N$ and we shall denote it
by
 $P^r   \rightarrow     N$ with $\partial P^r   \rightarrow     \partial N$ a homeomorphism. Let $q: N   \rightarrow     B$ be the retraction.  It can 
 be made transverse to the dual cell-decomposition, the map $P^r   \rightarrow     N$ is in the natural way a surgery mock bundle (cf. Nicas~\cite[Section~3.2]{Nic82})

\begin{remark}
If conversely, we are
given a degree one normal map $(f,b): P^r   \rightarrow     Q^r$ with the reference map $q: Q^r   \rightarrow     B$, one can define an element $x \in H_r (B, \mathbb{L}^+)$ by splitting $(f,b)$ into pieces using transversality of $q$ with respect to the dual cell-decomposition of $B \subset S^m$.
\end{remark}

\subsection{The homomorphism \texorpdfstring{$H_n (B, \mathbb{L})   \rightarrow     H_n (B, L_0)$}{Hn (B,L) to Hn (B,L0)}}\label{subsec2.3}

Without loss of generality we may assume 
that
$\dim B = n$. Let $B^{(n-1)}$ be the $(n-1)$-skeleton of $B$. This implies that
$$H_n (B, \mathbb{L}) \cong Z_n (B) \otimes L_0 \hookrightarrow C_n (B) \otimes L_0 \cong H_n (B, B^{(n-1)}, L_0)$$
is injective. Here, $Z_n (B)$ are the $n$-cycles of $B$ and $C_n (B)$ are the $n$-chains. Moreover, from the Atiyah-Hirzebruch spectral sequence one  easily gets
that $$H_n (B, B^{(n-1)},\mathbb{L}) \xrightarrow{\cong} H_n (B, B^{(n-1)}, L_0).$$ 

\begin{lemma}
The natural map $$H_n (B, \mathbb{L})   \rightarrow     H_n (B, B^{(n-1)},\mathbb{L})$$ factorizes as $$H_n (B, \mathbb{L})   \rightarrow     H_n (B, L_0) \subset H_n (B, B^{(n-1)},L_0) \cong H_n (B, B^{(n-1)}, \mathbb{L}).$$
\end{lemma}

\begin{proof}
This follows by the
commutativity of the diagram:
$$
\CD
@>>> H_n (B, \mathbb{L})   @>>>     \operatorname{H_n (B, B^{(n-1)}, \mathbb{L})} @>>>\\
@.  @VVV                @VV \cong V\\
@>>> H_n (B, L_0)   @>>> \operatorname{H_n (B, B^{(n-1)},L_0)} @>>>\\
\endCD
$$
induced by the map of spectra $\mathbb{L}    \rightarrow     K(L_0, 0)$.
\end{proof}

To prepare the next lemma we must study the spectral sequence $$E^2_{pq} \cong H_p (B, L_q) \xRightarrow[p+q=m]{} H_m (B, \mathbb{L})$$ in more detail.
First, we note that $$E^\infty_{n,m-n} \subset E^2_{n,m-n},$$ since $H_p (B, L_q) = 0$ for $p > n$. Moreover, $$E^\infty_{n,m-n} = \sfrac{F_{n,m-n}}{F_{n-1,m-n+1}},$$ where
 $$F_{n,m-n} = \hbox{\rm Im} (H_m (B^{(n)}, \mathbb{L})    \rightarrow     H_m(B, \mathbb{L})) \cong H_m (B, \mathbb{L}).$$ We consider the composite map $$\alpha:  H_m (B, \mathbb{L})   \rightarrow     E^\infty_{n,m-n} \subset E^2_{n,m-n} \cong H_n (B, L_{m-n}) \cong Z_n (B) \otimes L_{m-n}.$$

\begin{lemma}\label{lem:diagram}
Let $B \subset S^m$, $\dim B = n$, and $m - n \equiv 0 (4)$. Then
$$
\CD
H_n (B, \mathbb{L})   @>>>     \operatorname{H_n (B, L_0) \cong Z_n (B) \otimes L_0}\\
@VV \cong V                @V \cong V \beta V\\
H_m (B, \mathbb{L})   @>> \alpha > \operatorname{Z_n (B) \otimes L_{m-n}}\\
\endCD
$$
commutes. Here, $$H_n (B, \mathbb{L}) \xrightarrow[]{\cong} H_m (B, \mathbb{L})$$ and $$\beta : Z_n (B) \otimes L_0 \xrightarrow[]{\cong} Z_n (B) \otimes L_{m-n}$$ are isomorphisms induced by periodicity.
\end{lemma}

The proof follows by
 the spectral sequences.\qed

We now describe
the image of $x \in H_n (B, \mathbb{L})$ in $$H_n (B, L_0) \cong Z_n (B) \otimes L_0 \subset C_n (B) \otimes L_0.$$
It can be written as $\sum k_\tau \cdot \tau$, where $\tau$ ranges over the $n$-simplices of $B$.
\begin{enumerate}
\item[Step 1.] Consider $x \in H_m (B, \mathbb{L}^+) \cong H_m (B, \mathbb{L}) \cong H_n (B, \mathbb{L})$.
\item[Step 2.] Represent $x$ as the
cycle $x: (S^m, S^m\setminus B)   \rightarrow     (\mathbb{L}_0, \emptyset)$.
\item[Step 3.] Consider $x(\tau) : (f_\tau, b_\tau): P^{m-n}_\tau   \rightarrow     Q^{m-n}_\tau$ for $\tau < B$, $|\tau| = n$.
\end{enumerate}

One observes that $\partial Q^{m-n}_\tau = \partial P^{m-n}_\tau = \emptyset$ because its boundaries are composed of
 elements $x (\rho)$, with $|\rho| > n$ (because the boundary $\partial D (\tau, S^m)$ is formed from cells of type $D(\rho, S^m)$, $|\rho| > n)$. Now $\dim B = n$, so $(f_\tau, b_\tau)$ is a closed surgery problem.

To summarize, we have obtained

\begin{corollary}\label{cor:2.4}
Let $\dim B = n$, $B \subset S^m$, with $m - n \equiv 0 (4)$. An element $x \in H_n (B, \mathbb{L})$ has the
image in $$H_n (B, L_0) \cong Z_n (B) \otimes L_0 \cong Z_n (B) \otimes L_{m-n}$$ equal to $$\sum_{\tau < B^{(n)}} n_\tau \tau$$ with $n_\tau = $ image of $\sigma (f_\tau, b_\tau)$ under $$L_{m-n} (\pi_1 (Q^{m-n}_\tau))   \rightarrow     L_{m-n} \cong L_0.$$
\end{corollary}

\noindent \textbf{Supplement to Lemma  \ref{lem:diagram} and Corollary  \ref{cor:2.4}.}

The diagram in Lemma  \ref{lem:diagram} can be rewritten as
$$
\CD
H_n (B, \mathbb{L})   @>>>     \operatorname{H_n (B, L_0)}\\
@VV \cong V                @VV \cong V  \\
H_m (B, \mathbb{L})   @>>> \operatorname{H_n (B, L_{m-n})}\\
\endCD
$$
where the map $$H_m (B, \mathbb{L})   \rightarrow     H_n (B, L_{m-n})$$
 is the composition of 
$$H_m (B, \mathbb{L}) \cong H_m (B, \mathbb{L} \langle m-n \rangle)$$
  (cf. Ranicki~\cite[p.~156]{Ran92}) and 
$$H_m (B, \mathbb{L} \langle m-n \rangle)   \rightarrow     H_n (B, L_{m-n})$$ (cf. Ranicki~\cite[p.~289]{Ran92}). 
Note
 also the following commutativity 
$$
\CD
L_n (B, \varepsilon, \delta)   @. \cong @. \operatorname{H_n (B, \mathbb{L})}\\
@VV \cong V       @.          @VV \cong V\\
L_m (B, \varepsilon, \delta)  @. \cong @. \operatorname{H_m (B, \mathbb{L}).}\\
\endCD
$$

The above calculation resulting in Corollary  \ref{cor:2.4} follows from
the compositions $$H_n (B, \mathbb{L})   \rightarrow     H_m (B, \mathbb{L})   \rightarrow     H_n (B, L_{m-n})$$ of the above diagrams.

For the other composition one has to determine the map $H_n (B, \mathbb{L})   \rightarrow     H_n (B, L_0)$. This was done by Ranicki (\cite{Ran92}). In Prop.~15.3(II) therein an explicit formula is established using however the algebraic version of the $\mathbb{L}$-spectrum. 
In fact, Proposition~15.3(II) is the formula for
 the case of the symmetric $\mathbb{L}$-spectrum, but it is similar for the quadratic $\mathbb{L}$-spectrum.

\section{\texorpdfstring{$H_n (B, \mathbb{L})$}{Hn (B, L)} as the
controlled Wall  group}\label{sec3}

We mentioned in Section  \ref{sec1} the controlled Wall  group $L_n (B, \varepsilon, \delta)$. It can be defined for any $n \geq 0$. As before, we
assume
 that $B$ is a finite polyhedron. 
  
 Based on the work of Yamasaki \cite{Yam98}, Quinn, Pedersen and   
Ranicki \cite{PedQuiRan03} proved the following result.
\begin{theorem}\label{theor:A}
For finite dimensional ANR's there is a morphism $H_n (B, \mathbb{L})   \rightarrow     L_n (B, \varepsilon, \delta)$ which is an isomorphism for suitable $\varepsilon >0$ and $ \delta >0$.
\end{theorem}

\begin{remark}
In the paper by Pedersen, Quinn and Ranicki~\cite{PedQuiRan03}, $\mathbb{L}$
is the spectrum of quadratic algebraic Poincar\'{e} ads, and the morphism mentioned above is an assembling map. The proof of the theorem consists 
of
 showing that an element 
of
 $L_n (B, \varepsilon, \delta)$ can be split into pieces giving an element 
of
 $H_n (B, \mathbb{L})$. Now, the algebraic $\mathbb{L}$-spectrum is homotopy equivalent to the geometric one (cf. Ranicki~\cite{Ran92}), so $H_n (B, \mathbb{L})$ can be considered as the controlled Wall  group. 
\end{remark}

%%%%%%%%%%%%%%%%%%%%
As in the classical surgery theory, the controlled version leads to the
controlled surgery sequence
(cf. Ferry~\cite[Theorem 1.1.]{Fer10}).
This involves the controlled structure set
for which one needs the ''stability
properties'' as proved in Ferry~\cite[Theorem 10.2]{Fer10}.

%%%%%%%%%%%%%%%%%%%%

We shall now present the main result of this paper - an alternative proof that  $H_n(B,\mathbb{L})$ is the obstruction group
for controlled surgery problems.
\begin{theorem}\label{theorB}
Let $(f, b) : M^n   \rightarrow     X^n$ be a degree one normal map between manifolds, $n\geq 5$, and $\pi : X^n   \rightarrow     B$ a $UV^1$-map. Then  an element $$\sigma^c (f, b) \in H_n (B, \mathbb{L})$$ is defined
so
 that $\sigma^c (f, b) = 0$ if and only if
 $(f,b)$ is normally cobordant to a $\delta$-homotopy equivalence, uniquely up to $\varepsilon$-homotopy.
\end{theorem}

\begin{remark}
Note that the $UV^1$-condition for $\pi$ is no restriction when $n\geq 5$. The
 theorem holds for $n = 4$, if the $UV^1$-condition is satisfied. 
\end{remark}

\begin{proof}
The map $\pi : X   \rightarrow     B$ can be assumed to be transverse to the dual cells of $B$ (cf. Cohen~\cite{Coh67}); i.e. $$\pi^{-1} (D (\sigma, B)) = X^{n-|\sigma|}_\sigma$$ is an $(n-|\sigma|)$-dimensional submanifold. If we embed $B \subset S^m$,
for
 $m$ sufficiently large, we have $$\pi^{-1} (D (\sigma, B)) = \pi^{-1} (D (\sigma, S^m)),$$ and $X^{n - |\sigma|}_\sigma$ has 
the
 corresponding $(m - |\sigma| +3)$-ad structure. By transversality we define $$M^{n - |\sigma|}_\sigma = f^{-1} (X^{n - |\sigma|}_\sigma).$$ The restrictions of $b$ gives a family $$\{ (f_\sigma, b_\sigma) : M^{n - |\sigma|}_\sigma   \rightarrow     X^{n - |\sigma|}_\sigma | \sigma \subset B \}$$ which obviously defines a cycle
 $$z: (S^m, S^m \setminus B)   \rightarrow     (\mathbb{L}_{n-m}, \emptyset),$$ i.e. an element $$[z] = \sigma^c (f,b) \in H_n (B, \mathbb{L}).$$

We now
suppose that $[z] = 0$, i.e. there is a simplicial map 
$$w: (S^m, S^m \setminus B) \times \Delta^1   \rightarrow     (\mathbb{L}_{n-m}, \emptyset)$$ 
with
 $w (0) = z,$ and $w (1) = \emptyset$ (cf. Ranicki~\cite[Section~12]{Ran92}). 
This means that the various $(m - |\sigma| +3)$-ads $M^{n - |\sigma|}_\sigma   \rightarrow     X^{n - |\sigma|}_\sigma$ normally bound. Since $\pi$ is $UV^1$, we can assume that these are simply-connected surgery problems. If $$\restr{f_\sigma}{\partial M_\sigma} : \partial M_\sigma   \rightarrow     \partial X_\sigma$$ is already a homotopy equivalence, it follows that
 $(f_\sigma, b_\sigma)$ is normally cobordant to a homotopy equivalence. The proof now proceeds by induction on
 $n - |\sigma|$.

Let 
$$X_q = \displaystyle\bigcup_{|\sigma| \geq q} X^{n - |\sigma|}_\sigma \ \ \hbox{\rm  and} \ \ M_q = \displaystyle\bigcup_{|\sigma| \geq q} M^{n - |\sigma|}_\sigma,$$ hence $$X_n \subset X_{n-1} \subset \ldots \subset X_1 \subset X_0 = X,$$ similarly for $M$.

\noindent \textit{The induction hypothesis:} The restriction $f$ to $M_q$ is a homotopy equivalence with 
the
inverse $\overline{f} : X_q   \rightarrow     M_q$ such that the homotopies of 
$$f \circ \overline{f} \sim Id_{X_q} \ \  \hbox{\rm and}   \ \ \overline{f}\circ f  \sim Id_{M_q}$$ 
are controlled, i.e. when restricted onto $X^{n - |\sigma|}_\sigma$ (resp. $M^{n - |\sigma|}_\sigma)$ they have 
tracks
 over $D (\sigma, B)$ when projected down to $B$. More precisely,
$$\restr{f}{M_\sigma} : M^{n - |\sigma|}_\sigma   \rightarrow     X^{n - |\sigma|}_\sigma$$ is a homotopy equivalence with the inverse 
$$\restr{\overline{f}}{X^{n - |\sigma|}_\sigma} : X^{n - |\sigma|}_\sigma   \rightarrow     M^{n - |\sigma|}_\sigma,$$ 
and the homotopies above restrict to homotopies of
$$\restr{f}{M_\sigma} 
\circ 
\restr{\overline{f}}{X_\sigma} 
\sim 
Id_{x_\sigma} 
\ \ \hbox{\rm  and} \ \  
\restr{\overline{f}}{X_\sigma} 
\circ 
\restr{f}{M_\sigma} 
\sim Id_{M_\sigma}$$ 
over $D (\sigma, B)$.

\noindent \textit{The inductive step:} 
Suppose we are given $\tau \subset B$ with $|\tau| = q-1$, i.e. $\dim X_\tau = \hbox{\rm dim} M_\tau = n - q + 1$, and $$\partial M_\tau = \displaystyle\bigcup_{\sigma} M_\sigma, \ \ \partial X_\tau = \displaystyle\bigcup_{\sigma} X_\sigma$$ with $|\sigma| = q$, and $\sigma$ a face of $\tau$.
By the inductive
 hypothesis, $\restr{f}{M_\sigma}$ is a homotopy
 equivalence. These can be glued together by the
well
known homotopy theory (cf. Hatcher~\cite{Hat02}, or Sullivan~\cite[Lemma~H]{Sul66}) to give a homotopy equivalence $\restr{f}{\partial M_\tau} : \partial M_\tau   \rightarrow     \partial X_\tau$. So let $$F_\tau : (V_\tau, M_\tau, M'_\tau)   \rightarrow     (X_\tau \times I, X_\tau \times 0, X_\tau \times 1)$$ be a normal cobordism as explained above such that $\restr{F_\tau}{M_\tau} = f_\tau$, $\restr{F_\tau}{M'_\tau} = f'_\tau$ 
are homotopy equivalences, and because surgery was
 done in the interior of $M_\tau$, we have that $$\restr{F_\tau}{\partial V_\tau} : \partial V_\tau = M_\tau \cup \partial M_\tau \times I \cup M'_\tau   \rightarrow     X_\tau \times \{0\} \cup \partial X_\tau \times I \cup X_\tau \times \{1\}$$ coincides with $$f_\tau \cup (f_\tau \times I) \cup f'_\tau$$ (note that $\restr{f'_\tau}{\partial M_\tau} = \restr{f_\tau}{\partial M_\tau} )$. 

We denote by $\overline{f}'_\tau :  X_\tau   \rightarrow     M_\tau$ a homotopy inverse of $f'_\tau$. 
In our construction we add the cylinders $\partial M_\tau \times I$ and $\partial X_\tau \times I$ to $M'_\tau$ and $X_\tau \times 1$, 
and again denote them by
 $M'_\tau$ and $X_\tau \times 1$.
Then $f$ and $f'_\tau$ can be glued to give a homotopy equivalence 
$$f \cup f'_{\tau} : M_q \cup M'_\tau   \rightarrow     X_q \cup X_\tau.$$

This can be done for every $\tau \subset B$ with $|\tau| = q - 1$. If $M_\tau \cap M_{\tau'}$ are nonempty, they intersect in a common face $M_\sigma$, resp. $X_{\sigma}$, where we have the map $f$. 
Glued together they give a homotopy equivalence $f' : M_{q - 1}   \rightarrow     X_{q - 1}$.

\begin{lemma}
%[Hilfsatz]
 There is a homotopy inverse $\overline{f}' : X_{q - 1}  \rightarrow    M_{q - 1}$ such that $\restr{\overline{f}'}{X_q} = \overline{f}$, and $\restr{\overline{f}'}{X_\tau}$ is a homotopy inverse of $f'_\tau$ for every $\tau \subset B$ with $|\tau| = q - 1$.
\end{lemma}

\begin{proof}
We fix $\tau \subset B$, $|\tau| = q - 1$. First note that $\restr{\overline{f}}{\partial X_\tau} \sim \restr{\overline{f}'_\tau}{\partial X_\tau}$ (where $\overline{f}'_\tau$ is the above introduced inverse of $f'_\tau)$. This can be seen as follows:
$$f \circ \restr{\overline{f}}{\partial X_\tau} \sim Id_{\partial X_\tau} \ \ \hbox{\rm and} \ \  f_\tau \circ \restr{\overline{f}'_\tau}{\partial X_\tau} \sim Id_{\partial X_\tau}$$
implies 
$$f_\tau \circ \restr{\overline{f}}{\partial X_\tau} = f \circ \restr{\overline{f}}{\partial X_\tau} \sim f_\tau \circ \restr{\overline{f}'_\tau}{\partial X_\tau}. $$
However,
 $f_\tau$ is a homotopy equivalence, hence $\restr{\overline{f}}{\partial X_\tau} \sim \restr{\overline{f}'_\tau}{\partial X_\tau}$.

Let $$H_t : \partial X_\tau  \rightarrow    \partial M'_\tau = \partial M_\tau$$ be a homotopy such that
$$H_0 = \restr{\overline{f}}{\partial X_\tau}
\ \ \hbox{\rm and} \ \ 
H_1 = \restr{\overline{f}'_\tau}{\partial X_\tau}.$$
 By the Homotopy Extension Property we obtain a homotopy 
$\widetilde{H}_t : X_\tau \times I   \rightarrow     M'_\tau$ such that
\begin{diagram}
X_\tau \times I & &  \\
\bigcup  & \rdTo^{\widetilde{H}_t} & \\
\partial X_\tau \times I \cup X_\tau \times \{1\} & 
\rTo_{H_t \cup \overline{f}'_\tau} & 
\partial M'_\tau \cup M'_\tau =  M'_\tau
\end{diagram}
commutes. 
Hence $$\widetilde{f}_\tau = \widetilde{H}_0 : X_\tau   \rightarrow     M'_\tau$$
 is a homotopy equivalence such that $\restr{\widetilde{f}_\tau}{\partial X_\tau} = \restr{\overline{f}}{\partial X_\tau}$.
Hence $$\overline{f} \cup \widetilde{f}_\tau : X_q \cup X_\tau   \rightarrow     M_q \cup M'_\tau$$ is a homotopy inverse of $f'$ and it has the desired property. 
Since
at the  intersection $X_\tau \cap X_{\tau'}$ the maps 
$\widetilde{f}_\tau, \widetilde{f}_{\tau'}$ coincide with $\overline{f}$, 
 we can glue them together to get
 $\overline{f}' : X_{q - 1}   \rightarrow     M_{q - 1}$ as claimed.
\end{proof} 

In order to complete the proof of Theorem  \ref{theorB} it remains to prove that there are homotopies of $f' \circ \overline{f}' \sim Id_{X_{q - 1}}$, and $\overline{f}' \circ f' \sim Id_{M_{q - 1}}$ with small tracks. We shall construct such a homotopy for  $f' \circ \overline{f}' \sim Id_{X_{q - 1}}$. The other case is similar.

We let $H_t : X_q \times I   \rightarrow     X_q$ be the homotopy of $f\circ \overline{f} \sim Id_{X_q}$ given by the inductive hypothesis, so 
$h_t =  \restr{H_t}{\partial X_\tau}$
 is a homotopy of $f \circ \restr{\overline{f}}{\partial X_\tau} \sim \restr{Id}{\partial X_\tau}$. Recall that $X_q \cap X_\tau = \partial X_\tau$, so $f'_\tau \circ \overline{f}'_\tau$ coincides with $h_0 = f_\tau \circ \overline{f}_\tau$ on $\partial X_\tau.$
We consider $$h_t \cup f'_\tau \circ \overline{f}'_\tau : \partial X_\tau \times I \cup X_\tau \times \{0\}   \rightarrow     X_\tau$$ and apply the Homotopy Extension Property to obtain $h'_t = X_\tau \times I    \rightarrow     X_\tau$ such that
\begin{diagram}
X_\tau \times I & &  \\
\bigcup  & \rdTo^{h'_t} & \\
\partial X_\tau \times I  \cup X_\tau \times \{0\} & \rTo & X_\tau
\end{diagram}

The map $h'_1 :  X_\tau   \rightarrow     X_\tau$ is homotopic to $Id_{X_\tau}$, since $h'_0 = f'_\tau \circ \overline{f}'_\tau$ and it satisfies $\restr{h'_1}{\partial X_\tau} = h_1 = Id_{\partial X_\tau}$.

It follows from
Hatcher~\cite[Proposition~0.19]{Hat02}
 that $h'_1$ is homotopic relative $\partial X_\tau$ to $Id_{X_\tau}$ by a homotopy $h''_t$ (note that
here
 $Id_{X_\tau}$ is a homotopy inverse of $h'_1)$. We can therefore compose the homotopies $h'_t$ and $h''_t$ in the usual way to get a homotopy $$(h' \ast h'')_t : X_\tau \times I   \rightarrow     X_\tau$$ which coincides with $H_t$ on $X_q \cap X_\tau$, giving a homotopy $$H_t \cup (h' \ast h'')_t : (X_q \cup X_\tau) \times I    \rightarrow     X_q \cup X_\tau$$ between $(f \circ \overline{f}) \cup (f'_\tau \circ f'_\tau)$  and $Id$.

If $\tau, \tau' \subset B$ are $(q + 1)$-simplices such that $X_\tau \cap X_{\tau'} \neq \emptyset$, they intersect in a common face $\sigma, |\sigma| = q$, so the above constructed homotopies coincide
 with $H_t$, i.e. we can glue them together to get the desired controlled homotopies.

One notes that
the
 tracks can be arbitrary small (measured in $B$) if we use an arbitrary small cell-decomposition of $B$. 
This proves the inductive step.

We have in particular to consider the low-dimensional cases $n$, $n-1$, and
$n-3$, because surgery does not apply (note that in dimension $4$ one has to apply Freedman's result).

By the degree one property we can assume that $M_n = X_n$.
For $n-i, \ 1 \leq i \leq 3$, the pieces $$(f_\tau, b_\tau) : M^j_\tau   \rightarrow     X^j_\tau, \ \  1 \leq j \leq 3,$$ are special.
Namely, $\partial X^j_\tau$ is a $(j - 1)$-sphere, because $\pi$ is $UV^1$. We can close $\partial X^j_\tau$ by a $j$-disk to get a closed simply-connected $j$-manifold, i.e. a $j$-sphere. By the inductive hypothesis, 
$\partial M^j_\tau$ 
must also be a $(j-1)$-sphere so $M^j_\tau$
can be closed. 

The closed problem $M^j_\tau   \rightarrow     X^j_\tau$ bounds a problem $W^{j+1}_\tau   \rightarrow     V^{j+1}_\tau$ (because $\sigma^c (f,b) = 0$). Deleting
the
 $(j+1)$-disks one obtains a normal cobordism between $$M^j_\tau   \rightarrow     X^j_\tau \ \ \hbox{\rm and} \ \  M'^j_\tau = S^j  \xrightarrow{\cong} X^j_\tau = S^j.$$

We can now choose a degree one map 
$$(V^{j+1}_\tau \setminus \mathring{D}^{j+1}, X^j_\tau, S^j)   \rightarrow     (S^j \times I, S^j \times \{0\}, S^j \times \{1\})$$
and obtain a composition
$$F_\tau : (W^{j+1}_\tau \setminus \mathring{D}^{j+1}, M^j_\tau, S^j)   \rightarrow     (X^j_\tau \times I, X^j_\tau \times \{0\}, X^j_\tau \times \{1\}).$$
With this $F_\tau$, the proof proceeds as above and Theorem~\ref{theorB} is finally proved.
\end{proof}

\section*{Epilogue}
We shall conclude this paper by a final remark on
the
 controlled Wall
realization.
In our earlier paper~\cite{Heg14}, we showed
that the controlled structure set of a manifold $X$ with control
map $q: X\rightarrow B$ is a subgroup of $H_{n+1} (B, X, \mathbb{L})$. The controlled Wall
action of $H_{n+1} (B, \mathbb{L})$ on it is then nothing but the canonical map $$H_{n+1} (B, \mathbb{L}) \rightarrow H_{n+1} (B, X, \mathbb{L})$$ of $\mathbb{L}$-homology groups.

\section*{Acknowledgements}
This research was supported by the Slovenian Research Agency grants P1-0292, J1-7025, J1-8131, N1-0064, and N1-0083. 
We thank K. Zupanc for her technical assistance with the preparation of the manuscript. 
We acknowledge the referee for comments and suggestions.

\end{document}